\documentclass[11pt]{amsart} \textwidth=14.5cm \oddsidemargin=1cm
\usepackage{mathrsfs}
\usepackage{tikz}
\usepackage{verbatim}
\usepackage{amsmath}
\usepackage{amsxtra}
\usepackage{amscd}
\usepackage{amsthm}
\usepackage{amsfonts}
\usepackage{amssymb}
\usepackage{eucal}
\usepackage{latexsym,todonotes}
\usepackage[linktocpage=true]{hyperref}
\hypersetup{colorlinks,linkcolor=blue,urlcolor=blue,citecolor=blue}
 
\usepackage[all]{xy}

\usepackage[linktocpage=true]{hyperref}

\begin{document}

\newtheorem{thm}{Theorem} [section]
\newtheorem{cor}[thm]{Corollary}
\newtheorem{lem}[thm]{Lemma}
\newtheorem{prop}[thm]{Proposition}
\newtheorem{conj}[thm]{Conjecture}
\newtheorem{claim}{Claim}[thm]

\theoremstyle{definition}
\newtheorem{definition}[thm]{Definition}
\newtheorem{example}[thm]{Example}

\theoremstyle{remark}
\newtheorem{rem}[thm]{Remark}

\numberwithin{equation}{section}

%Referring commands:
\newcommand{\thmref}[1]{Theorem~\ref{#1}}
\newcommand{\secref}[1]{Section~\ref{#1}}
\newcommand{\lemref}[1]{Lemma~\ref{#1}}
\newcommand{\propref}[1]{Proposition~\ref{#1}}
\newcommand{\corref}[1]{Corollary~\ref{#1}}
\newcommand{\remref}[1]{Remark~\ref{#1}}
\newcommand{\exref}[1]{Example~\ref{#1}}
\newcommand{\qbinom}[2]{\begin{bmatrix} #1\\#2 \end{bmatrix} }

%Simplified symbols:
\newcommand{\nc}{\newcommand}
 \nc{\Z}{{\mathbb Z}}
 \nc{\C}{{\mathbb C}}
 \nc{\N}{{\mathbb N}}
 \nc{\F}{{\mf F}}
 \nc{\KP}{\text{KP}}
 \nc{\Q}{\mathbb{Q}}
 \nc{\lad}{\lambda}
 \nc{\ep}{\epsilon}
 \nc{\h}{\mathfrak h}
 \nc{\n}{\mf n}
%\nc*{\la}{\lambda}
\nc{\La}{\Lambda}
 \nc{\is}{{\mathbf i}}
 \nc{\V}{\mf V}
 \nc{\bi}{\bibitem}
 \nc{\E}{\mc E}
 \nc{\ba}{\tilde{\pa}}
 \nc{\half}{\frac{1}{2}}
 \nc{\hgt}{\text{ht}\,}
 \nc{\mc}{\mathcal}
 \nc{\mf}{\mathfrak} 
 \nc{\hf}{\frac{1}{2}}
 \nc{\hz}{\hf+\Z}

\nc{\U}{\bold{U}}
\nc{\Ui}{\bold{U}^{\imath}}
\nc{\Udot}{\dot{\bold{U}}}
\nc{\Uidot}{\dot{\bold{U}^{\imath}}}
\newcommand{\tK}{\tilde{K}}
\newcommand{\EE}[1]{E^{{(#1)}}}
\newcommand{\FF}[1]{F^{{(#1)}}}
\newcommand{\dvev}[1]{{\mathfrak{t}}_{\ev}^{{(#1)}}}
\newcommand{\dvodd}[1]{{\mathfrak{t}}_{\odd}^{{(#1)}}}
\newcommand{\dv}[1]{{\mathfrak{t}}_{\odd}^{{(#1)}}}
\newcommand{\ggq}[1]{(1-q^{#1})}

\nc{\ov}{\overline}
\nc{\ul}{\underline}
\nc{\wt}{\widetilde}
\nc{\I}{\mathbb{I}}

\nc{\be}{e}
 \nc{\bff}{\bold{f}}
 \nc{\bk}{k}
 \nc{\bt}{t}
\nc{\id}{\text{id}}
\nc{\Ihf}{\I^\imath}
\nc{\one}{\bold{1}}
\nc{\fus}{\natural}
\nc{\fprime}{\bold{'f}}
\nc{\Qq}{\Q(q)}
\nc{\qq}{(q^{-1}-q)}
\nc{\uqsl}{\bold{U}_q({\mf{sl}_{2n+2}})}
\nc{\BLambda}{{\Lambda_{\inv}}}
\nc{\ThetaA}{\Theta}
\nc{\ThetaB}{\Theta^{\imath}}
\nc{\ThetaC}{\Theta^{\jmath}}
\nc{\B}{\bold{B}}
\nc{\Abar}{\psi}
\nc{\Bbar}{\psi_\imath}
\nc{\HBm}{\mc{H}_{B_m}}
\nc{\ua}{\mf{u}}
\nc{\nb}{u}
\nc{\wtA}{\texttt{wt}}

%bilinear form
\nc{\gq}{(1-q^{2})}
\nc{\A}{\mathcal A}

%color
\newcommand{\blue}[1]{{\color{blue}#1}}
\newcommand{\red}[1]{{\color{red}#1}}
\newcommand{\green}[1]{{\color{green}#1}}
\newcommand{\white}[1]{{\color{white}#1}}

%%%
%%%
\title[PBW bases for modified quantum groups]{PBW bases for modified quantum groups}
 
\author[Weiqiang Wang]{Weiqiang Wang}
\address{Department of Mathematics, University of Virginia, Charlottesville, VA 22904}
\email{ww9c@virginia.edu}

\subjclass[2010]{Primary 17B37}  
\keywords{Quantum groups, canonical basis, PBW basis}

\begin{abstract}
We construct a basis for a modified quantum group of finite type, extending the PBW bases of positive and negative halves of a quantum group. Generalizing Lusztig's classic results on PBW bases, we show that this basis is orthogonal with respect to its natural bilinear form (and hence called a PBW basis), and moreover, the matrix for the PBW-expansion of the canonical basis is unital triangular. All these follow by a new construction of the modified quantum group of arbitrary type, which is built on limits of sequences of elements in tensor products of lowest and highest weight modules. Explicitly formulas are worked out in the rank one case. 
\end{abstract}

%\vspace{.3cm}
\maketitle

% \setcounter{tocdepth}{1}    \tableofcontents

%\vspace{1em}

\begin{quote}
\begin{center}
{\em Dedicated to  George Lusztig with admiration and appreciation 
}
\end{center}
\end{quote}

\vspace{1em}

%%%%%
%%%%%
\section{Introduction}

\subsection{}

Let $\U =\U^-\U^0\U^+$ be a Drinfeld-Jimbo quantum group (of finite type). There are several remarkable bases for $\U^-$, known as PBW basis and canonical basis. The PBW basis is orthogonal with respect to a natural bilinear form on $\U^-$, providing an approach to the construction of the canonical basis \cite{Lu90, Lu91} (also cf. \cite{K91} for another construction of the canonical basis). Lusztig \cite{Lu92, Lu93} further constructed a canonical basis for the modified quantum group $\Udot$, which is compatible with canonical bases on the tensor products of lowest and highest weight modules ${}^\omega L(\lad) \otimes L(\mu)$, for various dominant weights $\lad, \mu \in X^+$. The canonical bases admit remarkable positivity properties in ADE and symmetric types; they have major impacts in several (geometric, combinatorial, and categorical) directions of representation theory. 

\subsection{}

The goal of this paper is to formulate a PBW basis for the modified quantum group $\Udot$ of finite type with an orthogonality property and to establish its relations to canonical basis. The formulation follows by a new construction for the modified quantum group of arbitrary type, which is built on limits of sequences of elements in tensor products of lowest and highest weight modules.

\subsection{}

The inspiration came from the computation by the author \cite{Wa21} of an orthogonal basis for an $\imath$quantum group of rank 1 in terms of $\imath$divided powers (aka $\imath$canonical basis) \cite{BW18, BeW18}; in this setting, the $\imath$quantum group is a polynomial algebra in one variable. The $\imath$quantum groups arise from quantum symmetric pairs, and as we view quantum groups as $\imath$quantum groups of diagonal type, it is natural to explore the counterpart in the Drinfeld-Jimbo quantum group setting, starting again in rank 1. 

Indeed, by imposing the orthogonality condition, we are able to construct naturally and compute explicitly an (apparently new) PBW basis for $\Udot$ of rank 1 in terms of the canonical basis. Moreover, relations of these two bases when acting on the tensor product modules of the form ${}^\omega L(p) \otimes L(p+m)$ can be further observed; compare \cite[\S25.3]{Lu93}. A (PBW) basis of $\Udot$ with an orthogonality property emerges as a limit in a suitable sense of the tensor products of PBW basis elements acting on the tensor product modules. The rank 1 case is carried out in \S~\ref{sec:rank1}--\ref{sec:rank1c}, and some reader might prefer to go over the rank 1 case first. A bilinear pairing formula between canonical basis elements which appeared in Lauda \cite[Proposition~ 2.8]{La10} (who gave a long combinatorial proof) follows most naturally from the orthogonality of PBW basis for $\Udot$ and the PBW expansion of the canonical basis of $\Udot$. 

We develop in Section~\ref{sec:PBW} a framework for studying the limits of the so-called standard sequences of elements in tensor products of lowest and highest weight modules, ${}^\omega L(\lad) \otimes L(\lad +\zeta)$ with $\zeta$ fixed, as $\lad$ tends to $\infty$. This leads to a $\Q(q)$-linear isomorphism (valid for quantum groups of arbitrary type)
\[
\mathcal F_\zeta\colon \U^+\otimes \U^- \stackrel{\cong}{\longrightarrow} \Udot \one_\zeta,
\]
which allows us to transfer any pair of bases for $\U^\pm$ to a new basis for $\Udot$. 

A {\em fused canonical basis} for $\Udot$ is obtained via $\mathcal F_\zeta$ (for various $\zeta$) from the pure tensors of canonical bases in $\U^+\otimes \U^-$. 

\subsection{}

Now assume $\U$ is of finite type. 
The observations made in the rank 1 example suggest us to define the PBW basis for $\Udot$ as the transfer under $\mathcal F_\zeta$ (for various $\zeta$) of a tensor product of PBW bases in $\U^+ \otimes \U^-$; here the PBW bases of $\U^\pm$ can be associated to any reduced expression of the longest Weyl group element $w_0$. We show that the PBW basis for $\Udot$ is orthogonal with respect to the standard bilinear form on $\Udot$; it contains as a subset the PBW bases for $\U^+$ and $\U^-$. We further show that the transition matrix from the canonical basis to the PBW basis on $\Udot$ is unital triangular. %The construction remains valid in affine type, except that we cannot claim the orthogonality of the resulting PBW basis of $\Udot$ as this fails for $\U^+$; see \cite{BCP99}.

Let us make clear that the PBW basis and the fused canonical basis for $\Udot$ are bases over $\Q(q)$, but they do not lie in the integral form of $\Udot$. Already in the rank 1 case, the coefficients of the PBW-expansion of the canonical basis are typically (up to factors of $q$-powers) rational functions of the following form 
\[
\prod_{a=1}^m \frac1{1-q^{-2a}}.
\]
These rational functions expand as power series in $q^{-1}$ with {\em positive integral} coefficients. 

For a general ADE type, we show that the canonical basis is PBW-positive with {\em positive} coefficients in $\N[[q^{-1}]] \cap \Q(q)$. The proof relies on two positivity results. To that end, we show that the canonical basis has an expansion with coefficients in $\N[[q^{-1}]] \cap \Q(q)$ in terms of the fused canonical basis by using a positivity result in finite type of Webster \cite[Corollary 8.9]{We15} on the expansion of any canonical basis element in terms of pure tensors of canonical basis elements in ${}^\omega L(\lad) \otimes L(\lad +\zeta)$. On the other hand, the canonical basis of $\U^+$ is PBW-positive; this was first established in \cite[Corollary~ 10.7]{Lu90} when the reduced expressions of $w_0$ are ``adapted",  conjectured by Lusztig and proved by Syu Kato \cite{Ka14} using categorification for arbitrary reduced expressions of $w_0$ (see \cite{BKM14} for a second categorification proof and \cite{Oy18} for another proof based on the positivity of canonical bases under comultiplication \cite{Lu91}). Consequently, it follows that the fused canonical basis of $\Udot$ is PBW-positive. 

It will be very interesting to explore if the fused canonical basis and the PBW basis of $\Udot$ admit a categorification in a generalized KLR categorical setting, generalizing the geometrical and categorical interpretation for the canonical basis and PBW basis of $\U^+$  \cite{KL10, R12, VV11, Ka14, BKM14}. 
We shall return in \cite{Wa21} to construct PBW bases for modified $\imath$quantum groups arising from quantum symmetric pairs. 

\vspace{2mm}
{\bf Acknowledgement.} The author is partially supported by the NSF grant DMS-2001351. We thank the referee for helpful comments.

%%%%%
%%%%%
\section{A limit construction of the modified quantum group}
\label{sec:PBW}

In this section, we develop a framework for studying limits of sequences of elements in tensor product $\U$-modules ${}^\omega L(\lad) \otimes L(\lad+\zeta)$, as $\lad$ tends to $\infty$. This leads to a linear isomorphism $\U^+ \otimes \U^- \rightarrow \Udot \one_\zeta$, which allows to construct new bases for $\Udot$, including the fused canonical basis of $\Udot$ arising from the pure tensors of canonical bases in $\U^+ \otimes \U^-$. 

\subsection{Quantum groups and bilinear forms}

We denote by $\U$ the quantum group \cite{Lu93} associated to the Cartan/root datum $(X, Y, \I,  \cdot)$ with a perfect bilinear pairing $\langle \cdot, \cdot \rangle : Y\times X \rightarrow \Z$; it is a $\Qq$-algebra generated by $E_i$, $F_i$, $K_\mu$, for $i\in \I, \mu \in Y$. Denote by $\Udot$ the modified quantum group \cite[Chapter~ 23]{Lu93}. 
Denote by $\FF{m}_i =F_i^m /[m]_{q_i}!$, for $i\in \I, m\in \N$, the divided powers of $F_i$. Denote by $X^+$ the set of dominant weights in $X$. The comultiplication $\Delta$ satisfies
\begin{align}
  \label{eq:Delta}
\Delta(F_i) = F_i \otimes \tK_i^{-1} + 1\otimes F_i,
\qquad
\Delta(E_i) = E_i \otimes 1 + \tK_i \otimes E_i. 
\end{align}
By identifying $\bff \cong \U^-$ ($z \mapsto z^-$), we have $\Qq$-linear maps ${}_i r, r_i : \U^- \rightarrow \U^-$, for $i\in \I$; cf. \cite[1.2.13]{Lu93}. 
We have \cite[3.1.6]{Lu93}, for $y \in \U^-$, 
\begin{align}
  \label{Eiy}
E_iy -y E_i = \frac{\tK_i \, {}_i r(y) - r_i(y) \tK_i^{-1}}{q_i -q_i^{-1}}. 
\end{align}
Similarly, by identifying $\bff \cong \U^+$ ($z \mapsto z^+$), we have $\Qq$-linear maps ${}_i r, r_i : \U^+ \rightarrow \U^+$, and for $x \in \U^+$, 
\begin{align}
  \label{Fiy}
F_ix -x F_i = \frac{\tK_i^{-1} \, {}_i r(x) - r_i(x) \tK_i}{q_i -q_i^{-1}}. 
\end{align}
Note that $\U^\pm$ are $\N \I$-graded: $\U^+ =\sum_{\nu \in \N\I} \U^+_\nu$, and $\U^+ =\sum_{\nu \in \N\I} \U^-_\nu$. We say an element $x \in \U^+$ (respectively, $y \in \U^-$) is homogeneous if $x \in \U^+_\nu$ (respectively, $y \in \U^-_\nu$) for some $\nu$. In this case, we denote $|x|=\nu$ and $|y|=-\nu$. 

Denote by $\B$ the canonical basis for $\mathbf f$; the isomorphism $\mathbf f \cong \U^\pm$ induces the canonical bases $\B^\pm$ ($b \mapsto b^\pm$) for $\U^\pm$. Let $\U^+_{\Z[q^{-1}]}$ (respectively, $\U^-_{\Z[q^{-1}]}$) denote the $\Z[q^{-1}]$-span of the canonical basis in $\U^+$  (respectively, $\U^-$). 

Let $L(\lad)$ be the highest weight $\U$-module with highest weight vector $\eta_\lad$ of highest weight $\lad\in X^+$, and let ${}^\omega L(\lad)$ be the lowest weight $\U$-module with lowest weight vector $\xi_{-\lad}$ of lowest weight $-\lad$. For $x\in \U^+$, $y \in \U^-$, $\lad \in X^+$ and $\zeta \in X$ such that $\lad +\zeta \in X^+$, it follows by \eqref{eq:Delta} that 
\begin{align}
  \label{eq:y}
y(\xi_{-\lad} \otimes  \eta_{\lad+\zeta}) &= \xi_{-\lad} \otimes y \eta_{\lad+\zeta},
\qquad
x(\xi_{-\lad} \otimes  \eta_{\lad+\zeta}) = x \xi_{-\lad} \otimes   \eta_{\lad+\zeta}.
\end{align}
There is an anti-involution $\rho$ on $\U$ such that, for $i\in \I, \nu \in Y$,
\[
\rho (E_i) =q_i \tK_i F_i, \quad
\rho (F_i) =q_i^{-1} E_i \tK_i^{-1}, \quad
\rho (\tK_\nu) =\tK_{-\nu}. 
\]
According to \cite{K91, Lu93}, there is a bilinear form $(\cdot, \cdot)$ on $L(\mu)$, for $\mu \in X^+$, such that $(\eta_\mu, \eta_\mu)=1$ and $(ux, y) =(x, \rho(u)y)$, for all $x,y \in L(\mu), u \in \U$; a bilinear form $(\cdot, \cdot)$ on ${}^\omega L(\mu)$ is defined similarly. A bilinear form $(\cdot, \cdot)$ on ${}^\omega L(\lad) \otimes L(\mu)$ is defined by $(x\otimes y, x'\otimes y') =(x,x') (y, y')$. 

% Denote by $\bold{W}^+$ a PBW basis of $\U^+$, which is a $\Z[q^{-1}]$-basis for $\U^+_{\Z[q^{-1}]}$, and by $\bold{W}^-$ a PBW $\Z[q^{-1}]$-basis for $\U^-_{\Z[q^{-1}]}$; see \cite{Lu93} for the finite type. Denote by $\mf c(b^+)$ the canonical basis element of $\U^+$ with a leading term $b^+ \in \bold{W}^+$, i.e., $\mf c(b^+)\equiv b^+ \pmod{\U^+_{\Z[q^{-1}]}}$, and by $\mf c(b^-)$ the canonical basis element of $\U^-$ with a leading term $b^- \in \bold{W}^-$, i.e., $\mf c(b^-)\equiv b^- \pmod{\U^-_{\Z[q^{-1}]}}$. Denote by $\mf c(b^+) \diamondsuit_{\zeta} \mf c(b^-)$ the canonical basis element \cite[25.2.1]{Lu93}. 

There exists a unique $\Q(q)$-bilinear form $(\cdot, \cdot)$ on $\Udot$  \cite[26.1.2]{Lu93}, which extends the one on $\U^- (\cong \mathbf{f})$  \cite[1.2.3, 1.2.5]{Lu93}
such that  
\begin{align}
 (\one_{\lad} x \one_\mu, \one_{\lad'} x' \one_{\mu'}) &=0, \text{ for all } x, x' \in \Udot, \text{ unless } \lad =\lad' \text{ and } \mu =\mu',
  \label{eq:ad1} \\
(ux, y) &=(x, \rho(u)y), \text{ for } x, y \in \Udot, u\in \U,
  \label{eq:ad2} \\
(f \one_\lad, f' \one_\lad) &=(f, f'), \text{ for } f, f' \in \U^-, \lad \in X.
 \label{eq:ad3}
 \end{align}
This bilinear form is symmetric. %We set $\|x\|^2 =(x, x)$, for $x\in \Udot$. 
 %The canonical basis on $\Udot$ is almost orthonormal with respect to this bilinear form.
Note
\begin{align}
\label{BFhalf}
(\FF{m}_i, \FF{n}_i)  =(\EE{m}_i, \EE{n}_i)  = \delta_{m,n} (q_i^{-2}; q_i^{-2})_m^{-1},
\end{align}
where we have denoted  
\begin{align}
 \label{eq:g}
(a; q_i^{-2})_m =  \prod_{s=0}^{m-1} (1-a q_i^{-2s}),
\qquad
(q_i^{-2}; q_i^{-2})_m =  \prod_{s=1}^{m} (1-q_i^{-2s}).
\end{align}

\subsection{Standard sequences}

We shall often deal with sequences of elements $\{z_\lad\}_{\lad \in X^+}$, where $z_\lad \in {}^\omega L(\lad) \otimes L(\lad+\zeta)$ is
a linear combination of elements of the form 
\begin{align} 
& u\one_\zeta  (\xi_{-\lad} \otimes \eta_{\lad+\zeta} ),
\quad \text{ for } u \in \Udot,
 \label{eq:spanning} 
 \\
& x  \xi_{-\lad} \otimes y \eta_{\lad+\zeta},
\quad \text{ for } (x, y) \in \U^+ \times \U^-,
 \label{eq:spanning2}
\\
&  \text{where } u \text{ and } (x, y)  \text{ run over some finite sets independent of $\lad$}.
 \label{eq:finite}
\end{align}
We shall refer to such a sequence {\em standard}. 
The coefficients of $z_\lad$ usually take a certain form, which we now specify. 

\begin{definition}
  \label{def:seq}
(1) 
A standard sequence $\{z_\lad\}_{\lad \in X^+}$ is said to be {\em bounded} if $z_\lad$ can be written as a linear combination of elements  \eqref{eq:spanning}--\eqref{eq:spanning2} subject to \eqref{eq:finite}, with coefficients being a finite sum of the form
\begin{align}   \label{eq:s1}
\sum_{s\ge 0} \sum_{\vec{i} =(i_1, \ldots, i_s) \in \I^s} f_{\vec{i}}(q) \prod_{a=1}^s q_{i_a}^{-2\langle {i_a}, \lad \rangle}, \text{ 
where $f_{\vec{i}}(q) \in \Q(q)$ is independent of $\lad$.}
\end{align}

(2) 
A standard sequence $\{z_\lad\}_{\lad \in X^+}$ is said to be {\em asymptotically zero}, if $z_\lad$ can be written as a linear combination of  elements \eqref{eq:spanning}--\eqref{eq:spanning2} subject to \eqref{eq:finite}, with coefficients being a finite sum of the form 
\begin{align}   \label{eq:s0}
\sum_{s\ge 1} \sum_{\vec{i} =(i_1, \ldots, i_s) \in \I^s} f_{\vec{i}}(q) \prod_{a=1}^s q_{i_a}^{-2\langle {i_a}, \lad \rangle}, \text{ 
where $f_{\vec{i}}(q) \in \Q(q)$ is independent of $\lad$.}
\end{align}
\end{definition}
We say $\lad$ tends to $\infty$ if $\langle i, \lad \rangle$ tends to $+\infty$, for each $i\in \I$; in this case we shall denote $\lad \mapsto \infty$. Note that the coefficients in \eqref{eq:s1} (respectively, \eqref{eq:s0}) converges in $\Q((q^{-1}))$ to some scalar (respectively, to $0$) as $\lad$ tends to $\infty$. Given bounded standard sequences $\{z_\lad\}_{\lad \in X^+}, \{z_\lad '\}_{\lad \in X^+}$ and $\{z_\lad ''\}_{\lad \in X^+}$, we shall denote 
\[
z_\lad = o(1), \qquad \text{ and  }\;  z_\lad ' = z_\lad '' + o(1),
 %o(q_i^{-2\langle i, \lad \rangle} | i\in \I)
\]
if $\{z_\lad\}_{\lad \in X^+}$ is asymptotically zero and $\{z_\lad ' - z_\lad ''\}_{\lad \in X^+}$ is asymptotically zero, respectively.

\subsection{Approximations}

There are 2 types of elements in \eqref{eq:spanning}--\eqref{eq:spanning2}, and we shall understand how to approximate one another as $\lad \mapsto \infty$ in a precise fashion.

\begin{lem}
  \label{lem:yEi}
Let $x' \in \U^+, y \in \U^-$ be homogeneous, and $i\in \I$. Then we have
\begin{align}
&E_i (x' \xi_{-\lad} \otimes y \eta_{\lad+\zeta}) 
=  E_i x' \xi_{-\lad} \otimes y \eta_{\lad+\zeta} 
+ \frac{q_i^{\langle i, \zeta +|x'| +|y| \rangle +2}}{q_i -q_i^{-1}} x' \xi_{-\lad} \otimes   {}_i r(y)  \eta_{\lad+\zeta}
   \label{Eixy1}\\
&\qquad\qquad\qquad\qquad\quad
 - \frac{q_i^{-2\langle i, \lad \rangle} q_i^{\langle i,  -\zeta+|x'| \rangle } }{q_i -q_i^{-1}} x' \xi_{-\lad} \otimes r_i(y)  \eta_{\lad+\zeta}.
\notag
\end{align}
Equivalently, we have
\begin{align}
&E_i x' \xi_{-\lad} \otimes y \eta_{\lad+\zeta} 
= E_i (x' \xi_{-\lad} \otimes y \eta_{\lad+\zeta}) 
- \frac{q_i^{\langle i, \zeta +|x'| +|y| \rangle +2}}{q_i -q_i^{-1}} x' \xi_{-\lad} \otimes   {}_i r(y)  \eta_{\lad+\zeta}
   \label{Eixy2}\\
&\qquad\qquad\qquad\qquad\quad
+ \frac{q_i^{-2\langle i, \lad \rangle} q_i^{\langle i,  -\zeta+|x'| \rangle } }{q_i -q_i^{-1}} x' \xi_{-\lad} \otimes r_i(y)  \eta_{\lad+\zeta}.
\notag
\end{align}
\end{lem}

\begin{proof}
Using the comultiplication formula \eqref{eq:Delta} and the identity \eqref{Eiy}, we compute that 
\begin{align*}
&E_i (x' \xi_{-\lad} \otimes y \eta_{\lad+\zeta}) 
   \\
&= E_i x' \xi_{-\lad} \otimes y \eta_{\lad+\zeta} + \tK_i x' \xi_{-\lad} \otimes E_i y \eta_{\lad+\zeta}
 \notag \\
&= E_i x' \xi_{-\lad} \otimes y \eta_{\lad+\zeta} 
 + \tK_i x' \xi_{-\lad} \otimes \frac{\tK_i \, {}_i r(y) - r_i(y) \tK_i^{-1}}{q_i -q_i^{-1}} \eta_{\lad+\zeta}
\notag \\
&=  E_i x' \xi_{-\lad} \otimes y \eta_{\lad+\zeta} 
+ \frac{q_i^{\langle i, \zeta +|x'| +|y| \rangle +2}}{q_i -q_i^{-1}} x' \xi_{-\lad} \otimes   {}_i r(y)  \eta_{\lad+\zeta}
- \frac{q_i^{\langle i,  -2\lad-\zeta+|x'| \rangle } }{q_i -q_i^{-1}} x' \xi_{-\lad} \otimes r_i(y)  \eta_{\lad+\zeta}.
\notag
\end{align*}
This proves \eqref{Eixy1}. The identity \eqref{Eixy2} follows from \eqref{Eixy1}. 
\end{proof}

\begin{lem}
  \label{lem:xFi}
Let $x \in \U^+, y \in \U^-$ be homogeneous and $i\in \I$. Then we have
\begin{align}
&F_i (x \xi_{-\lad} \otimes y \eta_{\lad+\zeta}) 
=  x \xi_{-\lad} \otimes F_i y \eta_{\lad+\zeta}
+ \frac{q_i^{-\langle i, \zeta +|x| +|y| \rangle +2}}{q_i -q_i^{-1}} {}_i r(x) \xi_{-\lad} \otimes   y \eta_{\lad+\zeta}
   \label{Fixy1}\\
&\qquad\qquad\qquad\qquad\quad
- \frac{q_i^{-2\langle i,\lad\rangle} q_i^{\langle i, -\zeta-|y| \rangle } }{q_i -q_i^{-1}} r_i(x) \xi_{-\lad} \otimes y \eta_{\lad+\zeta}.
\notag
\end{align}
\end{lem}

\begin{proof}  
Using the comultiplication formula \eqref{eq:Delta} and the identity \eqref{Fiy}, we compute that 
\begin{align*}
&F_i (x \xi_{-\lad} \otimes y \eta_{\lad+\zeta}) 
   \\
&= x \xi_{-\lad} \otimes F_i y \eta_{\lad+\zeta}
 + F_i x \xi_{-\lad} \otimes \tK_i^{-1} y \eta_{\lad+\zeta}  \notag \\
&=   x \xi_{-\lad} \otimes F_i y \eta_{\lad+\zeta}
+ \frac{\tK_i^{-1} \, {}_i r(x) - r_i(x) \tK_i}{q_i -q_i^{-1}}
 \xi_{-\lad} \otimes \tK_i^{-1} y \eta_{\lad+\zeta} 
\notag \\
&=  x \xi_{-\lad} \otimes F_i y \eta_{\lad+\zeta}
+ \frac{q_i^{-\langle i, \zeta +|x| +|y| \rangle +2}}{q_i -q_i^{-1}} {}_i r(x) \xi_{-\lad} \otimes   y \eta_{\lad+\zeta}
- \frac{q_i^{\langle i,  -2\lad-\zeta-|y| \rangle } }{q_i -q_i^{-1}} r_i(x) \xi_{-\lad} \otimes y \eta_{\lad+\zeta}.
\notag
\end{align*}
The lemma is proved. 
\end{proof}
\begin{lem} 
  \label{lem:stable}
 Let $g \in \Udot$. 
 \begin{enumerate}
\item 
If $\{z_\lad\}_{\lad \in X^+}$
 is a bounded standard sequence, then so is $\{g z_\lad \}_{\lad \in X^+}$. 
\item
 If $\{z_\lad\}_{\lad \in X^+}$ is an asymptotically-zero standard sequence, then so is $\{g z_\lad \}_{\lad \in X^+}$. (We shall write $g \cdot o(1) =o(1)$.)

 \end{enumerate}
\end{lem}

\begin{proof}
We shall prove (1) only, as the proof of (2) is entirely similar.

If $g=g_1 +g_2$ and if $g_i z_\lad$ is a bounded standard sequences (for $i=1,2$), then so is $g z_\lad$. So we can assume $g =E_{i_a} \cdots E_{i_1} F_{j_b} \cdots F_{j_1} \one_\zeta$, and we only need to concern about $g$ acting on elements of the form \eqref{eq:spanning2}. A simple induction on $a+b$ reduces the proof to two basic cases for $g =E_i  \one_\zeta$ and $g =F_i \one_\zeta$, which in turn follow by applying the identities \eqref{Eixy1} and \eqref{Fixy1}, respectively. 
\end{proof}

Given $\mu =\sum_{i\in \I} n_i i \in \N\I$, we define its height $\hgt \mu =\sum_i n_i$. Recall $\B$ is the canonical basis of $\bff$, and any $b\in \B$ has weight $|b| \in \N \I$. 
Given $N, N'\in \N$, we shall denote by $P^+(N)$ (respectively, $P^-(N)$) the $\Q(q)$-submodule of $\U^+$ (respectively, $\U^-$) spanned by the elements $b^+ \in \B^+$ (respectively, $b^- \in \B^-$) , for $b\in \B$ such that $\hgt |b| \le N$. For $\zeta \in X$, we denote by $P(N, N')$  the $\Q(q)$-submodule of $\Udot$ spanned by the elements $b_1^+b_2^-\one_\zeta$, where $b_1, b_2 \in \B$ are such that $\hgt |b_1| \le N, \hgt |b_2| \le N'$ and $|b_1| - |b_2| =\zeta$. The following is the most crucial technical construction in this paper. 

\begin{prop}
  \label{prop:approx}
  Let  $\zeta \in X$.
 \begin{enumerate}
 \item
Given $x \in \U^+, y \in \U^-$, there exists a unique element $x \fus_\zeta y   \in \Udot \one_\zeta$ such that 
\begin{align}  
  \label{starxy}
 (x \fus_\zeta y)  (\xi_{-\lad} \otimes  \eta_{\lad+\zeta} )
 -x \xi_{-\lad} \otimes y \eta_{\lad+\zeta} 
= o(1). 
\end{align}

\item
Given $u  \in \Udot \one_\zeta$, there exists a unique element $u'' =\sum_k x_{k} \otimes y_{k} \in \U^+ \otimes \U^-$ such that 
\begin{align} 
 \label{Eixy4}
u  ( \xi_{-\lad} \otimes  \eta_{\lad+\zeta} )
- \sum_{k} x_{k} \xi_{-\lad} \otimes y_{k} \eta_{\lad+\zeta}
= o(1). 
\end{align}
\end{enumerate}
\end{prop}

\begin{proof}
(1) 
For the existence, we shall prove the following more precise statement. 

{\bf Claim 1.} For $x \in \U^+$ and $y \in \U^-$ homogeneous, there exists $x \fus_\zeta y \in P(\hgt |x|, \hgt |y|)$ of the form 
\begin{align}  \label{starxy1}
x \fus_\zeta y \in xy \one_\zeta + P(\hgt |x|-1, \hgt |y| -1)
\text{ such that \eqref{starxy} holds. 
}
\end{align}

We prove Claim~1 by induction on $\hgt |x|$. The case when $\hgt |x|=0$ clearly follows by \eqref{eq:y}. If $\hgt |x| >0$, we can assume $x$ is of the form $x =E_{i} x'$, for some $x' \in \U^+$. We observe by \eqref{Eixy2} that 
\begin{align}   \label{Eixy3} 
&E_i x' \xi_{-\lad} \otimes y \eta_{\lad+\zeta} 
= E_i (x' \xi_{-\lad} \otimes y \eta_{\lad+\zeta}) 
- \frac{q_i^{\langle i, \zeta +|x'| +|y| \rangle +2}}{q_i -q_i^{-1}} x' \xi_{-\lad} \otimes  {}_i r(y)  \eta_{\lad+\zeta}
   +  o(1).
\end{align}
Since $\hgt |x'| < \hgt |x|$, the inductive assumption can be applied to $x' \xi_{-\lad} \otimes y \eta_{\lad+\zeta}$ and $x' \xi_{-\lad} \otimes {}_i r(y) \eta_{\lad+\zeta}$ on the RHS\eqref{Eixy3}, and we have 
\begin{align*}
x' \xi_{-\lad} \otimes y \eta_{\lad+\zeta}
&= (x' \fus_{\zeta} y) ( \xi_{-\lad} \otimes y \eta_{\lad+\zeta} ) + o(1),
\\
x' \xi_{-\lad} \otimes {}_i r(y) \eta_{\lad+\zeta}
&= (x' \fus_{\zeta} {}_i r(y)) ( \xi_{-\lad} \otimes y \eta_{\lad+\zeta} ) + o(1). 
\end{align*}
By Lemma~\ref{lem:stable}, $E_i \cdot o(1)$ is of the form $o(1)$. Hence Equation \eqref{Eixy3} can be written as 
\begin{align*}
E_i x' \xi_{-\lad} \otimes y \eta_{\lad+\zeta}
&= \Big( E_i (x' \fus_\zeta y)  
- \frac{q_i^{\langle i, \zeta +|x'| +|y| \rangle +2}}{q_i -q_i^{-1}} 
x' \fus_\zeta  {}_i r(y) \Big)   ( \xi_{-\lad} \otimes   \eta_{\lad+\zeta} )
+  o(1). 
\end{align*}
By the inductive assumption, we have 
\[
E_i(x' \fus_\zeta y) \in E_i(x' y \one_\zeta + P(\hgt |x'|-1, \hgt |y| -1) ) \subseteq xy \one_\zeta + P(\hgt |x|-1, \hgt |y| -1)
\]
and $x' \fus_\zeta {}_i r(y) \in P(\hgt |x|-1, \hgt |y| -1)$. Therefore, setting 
\[
x \fus_\zeta y:= E_i (x' \fus_\zeta y)  - \frac{q_i^{\langle i, \zeta +|x'| +|y| \rangle +2}}{q_i -q_i^{-1}} 
x' \fus_\zeta  {}_i r(y)
\]
will satisfy \eqref{starxy} and \eqref{starxy1}. 

Assume that another element $w \in \Udot \one_\zeta$ also satisfies the same property \eqref{starxy} as $x \fus_\zeta y$. Set $z:=w -x \fus_\zeta y,$ and $z_\lad :=z (\xi_{-\lad} \otimes  \eta_{\lad+\zeta})$. Then $z_\lad = o(1)$ by using \eqref{starxy} twice. Thus $(z_\lad, z_\lad)$ converges to $0$, as $\lad$ tends to $\infty$. On the other hand, by \cite[26.2.3]{Lu93}, we have $(z_\lad, z_\lad)$ converges to $(z, z)$, as $\lad$ tends to $\infty$. 
Hence we must have $(z, z)=0$, and whence $z=0$, i.e., $w =x \fus_\zeta y$.

(2) 
For the existence, we shall prove a more precise statement.

{\bf Claim 2.} For $u =xy \one_\zeta \in \Udot \one_\zeta$ with $x \in \U^+$ and $y \in \U^-$ homogeneous, there exists $u'' \in P^+(\hgt |x|) \otimes P^-(\hgt |y|)$ of the form 
\begin{align}  \label{starxy2}
u'' \in x \otimes y  + P^+(\hgt |x|-1) \otimes P^- (\hgt |y| -1)
\text{ such that \eqref{Eixy4} holds. 
}
\end{align}

We prove the existence by induction on $\hgt |x|$. If $|x| =0$, then the statement follows by \eqref{eq:y}. Assume $\hgt |x| >0$, we can assume $x =E_i x'$, for some $x' \in \U^+$, and so $u =E_i x' y\one_\zeta$. 
%By the inductive assumption,
%\begin{align*} 
%x' \one_\zeta ( \xi_{-\lad} \otimes y \eta_{\lad+\zeta} )
%=\sum_{k} x'_{k} \xi_{-\lad} \otimes y'_{k} \eta_{\lad+\zeta} + o(1). 
%\end{align*}
Hence, by the inductive assumption on $x' y \one_\zeta$ (and recalling $E_i \cdot o(1) =o(1)$ by Lemma~\ref{lem:stable}), we have
\begin{align}
  \label{eq:Eiu2}  
u ( \xi_{-\lad} \otimes \eta_{\lad+\zeta} )
& = E_i \cdot x'y \one_\zeta ( \xi_{-\lad} \otimes \eta_{\lad+\zeta} )
\\
&= E_i (x' \xi_{-\lad} \otimes y \eta_{\lad+\zeta} + \sum_{\ell} x'_{\ell} \xi_{-\lad} \otimes y'_{\ell} \eta_{\lad+\zeta})
+ o(1),
 \notag
\end{align}
for some $\sum_{\ell} x'_{\ell} \otimes y'_{\ell} \in P^+(\hgt |x|-2) \otimes P^-(\hgt |y| -1)$. 
By applying \eqref{Eixy1},  we see that RHS\eqref{eq:Eiu2} $=u'' ( \xi_{-\lad} \otimes \eta_{\lad+\zeta} ) +o(1)$, for $u''$ of the form \eqref{starxy2}. 
The uniqueness can be established by the same arguments as in (1). 
\end{proof}

\subsection{To infinity}

Recall Definition~\ref{def:seq} for (asymptotically-zero) bounded standard sequences.

\begin{lem}
  \label{lem:lim0}
Let $\zeta \in X$, and let $\{z_\lad \}_{\lad \in X^+} \in {}^\omega L(\lad) \otimes L(\lad+\zeta)$ be a bounded standard sequence. Then,
$(z_\lad, u(\xi_{-\lad} \otimes \eta_{\lad+\zeta})) \in \Q(q)$ (and respectively, $(z_\lad, x \xi_{-\lad} \otimes y\eta_{\lad+\zeta}) \in \Q(q)$) converges in $\Q((q^{-1}))$ as $\lad $ tends to $\infty$, for any $u \in \Udot$, $x \in \U^+$, and $y \in \U^-$.

Moreover, the following statements (a)--(d) for $\{z_\lad\}_{\lad \in X^+}$ are equivalent:
\begin{enumerate}
%\item $(z_\lad, z_\lad)$ converges in $\Q((q^{-1}))$ to $0$ as $\lad$ tends to $\infty$;
\item[(a)]
$\{z_\lad \}_{\lad \in X^+}$ is asymptotically-zero;

\item[(b)]
$(z_\lad, u(\xi_{-\lad} \otimes \eta_{\lad+\zeta}))$ converges in $\Q((q^{-1}))$ to $0$ as $\lad$ tends to $\infty$, for any $u \in \Udot$;

\item[(c)]
$(z_\lad, x \xi_{-\lad} \otimes y\eta_{\lad+\zeta})$ converges in $\Q((q^{-1}))$ to $0$ as $\lad $ tends to $\infty$, for any $x \in \U^+$ and $y \in \U^-$;

\item[(d)]
$(z_\lad, z'_{\lad})$ converges in $\Q((q^{-1}))$ to $0$ as $\lad $ tends to $\infty$, for any bounded standard sequence $\{z'_\lad\}_{\lad \in X^+}$.
\end{enumerate}
\end{lem}
\noindent  If one of the conditions (a)--(d) above is satisfied, we say $\lim_{\lad \mapsto \infty}\limits z_\lad =0$.

\begin{proof}
Let us prove that $(z_\lad, u(\xi_{-\lad} \otimes \eta_{\lad+\zeta})) \in \Q(q)$ converges in $\Q((q^{-1}))$ as $\lad $ tends to $\infty$. Note $(z_\lad, u(\xi_{-\lad} \otimes \eta_{\lad+\zeta})) = (\rho(u) z_\lad, \xi_{-\lad} \otimes \eta_{\lad+\zeta})$, and $\{\rho(u)z_\lad\}$ is a bounded standard sequence spanned by elements of the form \eqref{eq:spanning}--\eqref{eq:spanning2} with coefficients $f_\lad(q)$ as in \eqref{eq:s1}. If $f_\lad(q) u'(\xi_{-\lad} \otimes \eta_{\lad+\zeta})$ is a summand of $\rho(u)z_\lad$, then $(f_\lad(q) u'(\xi_{-\lad} \otimes \eta_{\lad+\zeta}), \xi_{-\lad} \otimes \eta_{\lad+\zeta})$ converges in $\Q((q^{-1}))$ to $\lim\limits_{\lad \mapsto \infty} f_\lad(q) \cdot (u', \one_\zeta)$, by \cite[26.2.3]{Lu93}. 
If $f_\lad(q) x'\xi_{-\lad} \otimes y'\eta_{\lad+\zeta}$ is a summand of $\rho(u)z_\lad$, then 
$(f_\lad(q) x'\xi_{-\lad} \otimes y'\eta_{\lad+\zeta}, \xi_{-\lad} \otimes \eta_{\lad+\zeta})$ converges in $\Q((q^{-1}))$ to 
$\lim\limits_{\lad \mapsto \infty} f_\lad(q) \cdot (x',1)(y',1)$. Summarizing, $(z_\lad, u(\xi_{-\lad} \otimes \eta_{\lad+\zeta})) =(\rho(u) z_\lad, \xi_{-\lad} \otimes \eta_{\lad+\zeta})$ converges in $\Q((q^{-1}))$. 

Assume (a) holds. Then (b) and (c) follow by the same arguments above together with $\lim\limits_{\lad \mapsto \infty} f_\lad(q)=0$, and subsequently (d) also follows. Therefore, for any bounded standard sequences $\{z_\lad\}_{\lad \in X^+}, \{z'_\lad\}_{\lad \in X^+}$, we have 
\begin{align} 
  \label{eq:zero}
\lim\limits_{\lad \mapsto \infty}(z_\lad, z_\lad') =0, \quad 
\text{ if either sequence is }  o(1).
\end{align}

Now by Proposition~\ref{prop:approx}(1), $x \xi_{-\lad} \otimes y \eta_{\lad+\zeta} = (x\fus_\zeta y)
(\xi_{-\lad} \otimes \eta_{\lad+\zeta}) +o(1)$, and we have already shown that $\lim\limits_{\lad \mapsto \infty}(z_\lad, (x\fus_\zeta y) (\xi_{-\lad} \otimes \eta_{\lad+\zeta}))$ exists. Thus, by applying \eqref{eq:zero}, $\lim\limits_{\lad \mapsto \infty}(z_\lad, x \xi_{-\lad} \otimes y\eta_{\lad+\zeta}) = \lim\limits_{\lad \mapsto \infty}(z_\lad, (x\fus_\zeta y)
(\xi_{-\lad} \otimes \eta_{\lad+\zeta}))$ exists in $\Q((q^{-1}))$. 

The equivalence of (b) and (c) follows by \eqref{eq:zero} and Proposition~\ref{prop:approx}. Clearly Parts (b) and (c) are special cases of (d), and on the other hand,  Part (d) follows from (b)--(c) combined easily. 

It remains to show that (b) $\Rightarrow$ (a). By Proposition~\ref{prop:approx} and \eqref{eq:zero}, we can assume that $z_\lad$ is spanned by elements of the form \eqref{eq:spanning2}, i.e., $z_\lad =\sum_{u} f_{\lad,u} (q) u (\xi_{-\lad} \otimes \eta_{\lad+\zeta})$, for various nonzero $u$ which are orthogonal to each other. Then, for each such $u$, $(z_\lad, u(\xi_{-\lad} \otimes \eta_{\lad+\zeta}))$ converges  in $\Q((q^{-1}))$ to $\lim\limits_{\lad \mapsto \infty} f_{\lad,u} (q) \cdot (u,u)$, by \cite[26.2.3]{Lu93}. By the assumption (b), $\lim\limits_{\lad \mapsto \infty} f_{\lad,u} (q)=0$, so $f_{\lad,u} (q)$ is of the form \eqref{eq:s0} and $z_\lad$ is asymptotically zero. 

The lemma is proved. 
\end{proof}

We have the following reformulation of Proposition~\ref{prop:approx}. Given $x,y \in \bff$, we have $x^+ \in \U^+, y^- \in \U^-$, and we shall simply write $x \fus_\zeta y$ for $x^+ \fus_\zeta y^-$. 

\begin{prop}
  \label{prop:limit}
  Let $\zeta \in X$.
\begin{enumerate}
\item
Given $x, y \in \bff$, there exists a unique element $x \fus_\zeta y \in \Udot \one_\zeta$ such that 
\begin{align*}
\lim_{\lad \mapsto \infty}\limits  \Big((x \fus_\zeta y) (\xi_{-\lad} \otimes \eta_{\lad+\zeta} ) - x^+ \xi_{-\lad} \otimes y^- \eta_{\lad+\zeta} \Big ) =0.
\end{align*}

\item
Given $u \in \Udot \one_\zeta$, there exists a unique element $u''=\sum_k x_{k} \otimes y_{k} \in \U^+ \otimes \U^-$ such that 
\begin{align*}
\lim_{\lad \mapsto \infty}\limits  \Big(u (\xi_{-\lad} \otimes \eta_{\lad+\zeta} ) - \sum_{k} x_{k} \xi_{-\lad} \otimes y_{k} \eta_{\lad+\zeta} \Big ) =0.
\end{align*}
\end{enumerate}
\end{prop}

\subsection{A linear isomorphism for $\Udot \one_\zeta$}

We now present the first main result of this paper, which is built on the constructions in Proposition~\ref{prop:approx} or its reformulation in Proposition~\ref{prop:limit}; we retain the notation therein in Part (1) of the theorem below.

\begin{thm}  
  \label{thm:obasis}
{\quad}
\begin{enumerate}
\item
 For any $\zeta \in X$, there exists a $\Q(q)$-linear isomorphism 
$$
\mathcal F_{\zeta} \colon \U^+ \otimes \U^- \longrightarrow \Udot \one_\zeta
$$ 
such that 
\[
\mathcal F_\zeta  (x \otimes y ) =x \fus_\zeta y,
\qquad \mathcal F_\zeta ^{-1} (u) =u''.
\] 
\item
Given bases $B_1, B_2$ for $\bff$, the set 
\begin{align}
  \label{eq:PBW}
B_1 \fus B_2 := \big \{ b_1\fus_\zeta b_2 \mid b_1 \in B_1, b_2 \in B_2, \zeta \in X \big \}
\end{align}
forms a $\Q(q)$-basis for $\Udot$. Moreover, for $b_1, b_1' \in B_1, b_2, b_2' \in B_2$ and $\zeta, \zeta' \in X$, we have
\begin{align}  \label{eq:BF2}
\big(b_1 \fus_\zeta b_2, \, b'_1 \fus_{\zeta'} b'_2 \big) = \delta_{\zeta, \zeta'} (b_1,b'_1) \, (b_2,b'_2).
\end{align}
\end{enumerate}
\end{thm}

\begin{proof}
(1) 
Write $\mathcal F =\mathcal F_\zeta$, and denote $\mathcal G:  \Udot \one_\zeta   \longrightarrow \U^+ \otimes \U^-$ the linear map such that $\mathcal G(u) =u''$ as given in Proposition~\ref{prop:approx}(2) or Proposition~\ref{prop:limit}(2). It follows by the uniqueness in Proposition~\ref{prop:limit}(2) that $\mathcal G \mathcal F (x\otimes y)  =\mathcal G(x \fus_\zeta y) =x \otimes y$. 
Let $u \in \Udot \one_\zeta$ and $\mathcal G (u) =u'' :=\sum_k x_k \otimes y_k$ as in Proposition~\ref{prop:limit}(2). Then by Proposition~\ref{prop:limit}(1), we have $(\sum_k x_k \fus_\zeta y_k)  (\xi_{-\lad} \otimes  \eta_{\lad+\zeta} )
 -\sum_{k} x_k \xi_{-\lad} \otimes y_k \eta_{\lad+\zeta} = o(1).$ By the uniqueness in Proposition~\ref{prop:limit}(1), we must have $u= \sum_k x_k \fus_\zeta y_k$, i.e., $u=\mathcal F \mathcal G(u)$.
So $\mathcal F$ is a linear isomorphism with inverse $\mathcal G$.

(2)
The first statement on bases follows by (1). 
The formula \eqref{eq:BF2} is trivial if $\zeta' \neq \zeta$.  

Assume $\zeta'=\zeta$. 
%Let $b^+, b^+_1 \in \bold{W}^+, b^-, b^-_1 \in \bold{W}^-$.
Set 
\begin{align*}
z_\lad &:= (b_1 \fus_{\zeta} b_2) (\xi_{-\lad} \otimes \eta_{\lad+\zeta} ) - b_1 \xi_{-\lad} \otimes b_2 \eta_{\lad+\zeta},
\\
z'_\lad &:=  (b'_1 \fus_{\zeta} b'_2) (\xi_{-\lad} \otimes \eta_{\lad+\zeta} ) - b'_1 \xi_{-\lad} \otimes b'_2 \eta_{\lad+\zeta}. 
\end{align*}
It follows by Proposition~\ref{prop:limit} that  $\lim_{\lad \mapsto \infty}\limits z_\lad =0$ and $\lim_{\lad \mapsto \infty}\limits z'_\lad =0$. Then we have 
\begin{align*}
&\Big((b_1 \fus_{\zeta} b_2) (\xi_{-\lad} \otimes \eta_{\lad+\zeta} ), (b'_1 \fus_{\zeta} b'_2)  (\xi_{-\lad} \otimes \eta_{\lad+\zeta} ) 
\Big)
 - \big( b_1 \xi_{-\lad} \otimes b_2 \eta_{\lad+\zeta},
b'_1 \xi_{-\lad} \otimes b'_2 \eta_{\lad+\zeta} \big )
\\
&= \big(z_\lad, (b'_1 \fus_{\zeta} b'_2) (\xi_{-\lad} \otimes \eta_{\lad+\zeta} ) \big )
 + \big(z'_\lad, b_1 \xi_{-\lad} \otimes b_2 \eta_{\lad+\zeta} \big ),
\end{align*}
whose RHS converges in $\Q((q^{-1}))$ to $0$ as $\lad$ tends to $\infty$ by Lemma~\ref{lem:lim0}. 

By  \cite[26.2.3]{Lu93}, the pairing $\big((b_1 \fus_{\zeta} b_2) (\xi_{-\lad} \otimes \eta_{\lad+\zeta} ), (b'_1 \fus_{\zeta} b'_2)  (\xi_{-\lad} \otimes \eta_{\lad+\zeta} ) \big)$ converges in $\Q((q^{-1}))$ to $\big( b_1 \fus_{\zeta} b_2, b'_1 \fus_{\zeta} b'_2 \big)$ as $\lad$ tends to $\infty$. On the other hand, we have
\[
\big( b_1 \xi_{-\lad} \otimes b_2 \eta_{\lad+\zeta},
b'_1 \xi_{-\lad} \otimes b'_2 \eta_{\lad+\zeta} \big )
 =\big( b_1 \xi_{-\lad},
b'_1 \xi_{-\lad} \big ) \big(b_2 \eta_{\lad+\zeta}, b'_2 \eta_{\lad+\zeta} \big ),
\]
which converges to $(b_1, b'_1)(b_2, b'_2)$ as $\lad$ tends to $\infty$. Hence \eqref{eq:BF2} for $\zeta'=\zeta$ follows. 

The theorem is proved. 
\end{proof}
\subsection{Fused canonical basis for $\Udot$}

We have several distinguished choices of bases for $\bff$. If we choose the canonical basis $\B$ for $\bff$, the $\Q(q)$-basis 
\[
\B \fus \B = \big \{ b_1 \fus_\zeta b_2  \mid b_1, b_2\in \B, \zeta \in X \big \}
\]
is called a {\em fused canonical basis} for $\Udot$. Note that
\[
b^+ \one_\zeta \in \B \fus \B, \;
 b^- \one_\zeta \in \B \fus \B, 
\quad \text{ for } b \in \B. 
\]

Recall Lusztig's canonical basis $\{b_1 \diamondsuit_\zeta b_2 | b_1, b_2 \in \B, \zeta \in X \}$ on $\Udot$. 
Define a partial order $\leq$ on $\B \times \B$ as follows \cite[24.3.1]{Lu93}. We say $(b_1', b_2') \leq (b_1, b_2)$ if  $|b'_1| -|b'_2| =|b_1| -|b_2|$ and if either ($b_1 =b_1'$ and $b_2 =b_2'$) or 
\[
\hgt|b'_1| <\hgt|b_1| \text{ and } 
\hgt|b'_2|<\hgt|b_2|.
\]

\begin{thm}
  \label{thm:CB-fCB}
Let $\zeta \in X$, and $b_1, b_2 \in \B$. We have
\begin{align*}
b_1 \diamondsuit_\zeta b_2 
= b_1 \fus_\zeta b_2  
+ \sum_{(b'_1, b'_2) <(b_1, b_2)}  P_{(b'_1, b'_2),(b_1, b_2)}  (q) \, 
b'_1 \fus_\zeta b'_2
 \end{align*}
where $P_{(b'_1, b'_2),(b_1, b_2)} (q)   \in \Z [[q^{-1}]] \cap \Q(q)$.
Moreover, for the ADE type, $P_{(b_1', b_2'),(b_1, b_2)}^{\lad} (q)$ lies in $ \N[[q^{-1}]] \cap \Q(q)$. 
\end{thm}

\begin{proof}
As $\{b_1' \fus_\zeta b_2' \mid b_1', b_2' \in \B\}$ is a basis for $\Udot \one_\zeta$ by Theorem~\ref{thm:obasis}, we have
\begin{align}
  \label{eq:P}
b_1 \diamondsuit_\zeta b_2 = \sum_{(b'_1, b'_2)} 
P_{(b'_1, b'_2),(b_1, b_2)}  (q) \, 
b'_1 \fus_\zeta b'_2
 \end{align}
where $P_{(b'_1, b'_2),(b_1, b_2)} (q)   \in \Q(q)$. %It remains to show that these coefficients are $0$ unless $(b'_1, b'_2) \le (b_1, b_2)$, and $P_{(b_1, b_2),(b_1, b_2)} (q)=1$. 

There is a canonical basis $\{ (b_1 \diamondsuit b_2)_{\lad, \lad +\zeta} | b_1 \in \B(\lad), b_2 \in \B(\lad +\zeta) \}$ on ${}^\omega L(\lad) \otimes L(\lad +\zeta)$, for $\lad \in X^+$ and $\lad+\zeta \in X^+$, by \cite[24.3.3]{Lu93}. 
By \cite[25.2.1]{Lu93}, we have 
\begin{align} 
  \label{eq:CBCB}
(b_1 \diamondsuit_\zeta b_2) (\xi_{-\lad} \otimes \eta_{\lad +\zeta} )
= (b_1 \diamondsuit b_2)_{\lad, \lad +\zeta}, \text{ for } b_1 \in \B(\lad), b_2 \in \B(\lad +\zeta).
\end{align}

We shall assume $\lad$ is chosen to be large enough below such that $b_1 \in \B(\lad), b_2 \in \B(\lad +\zeta)$. 
 When combining \eqref{eq:CBCB} with \cite[24.3.3]{Lu93} on expansion of $(b_1 \diamondsuit b_2)_{\lad, \lad +\zeta}$ in terms of pure tensors of canonical bases on ${}^\omega L(\lad)$ and $L(\lad +\zeta)$, we have
\begin{align}
(b_1 \diamondsuit_\zeta b_2) (\xi_{-\lad} \otimes \eta_{\lad +\zeta} )
&= b_1^+ \xi_{-\lad} \otimes  b_2^- \eta_{\lad +\zeta}  
  \label{eq:Ptilde}\\
&\quad
+ \sum_{\stackrel{(b_1',b_2') \in \B(\lad) \times \B(\lad+\zeta)}{(b_1', b_2') < (b_1, b_2)} }
\widetilde P_{(b_1', b_2'),(b_1, b_2)}^{\lad} (q) \, 
b_1'^+ \xi_{-\lad} \otimes  b_2'^- \eta_{\lad +\zeta},
\notag
 \end{align}
 where $\widetilde P_{(b_1', b_2'),(b_1, b_2)}^{\lad} (q) \in \Z[q^{-1}]$, with the superscript indicating its dependence on $\lad$. We set $\widetilde P_{(b_1', b_2'),(b_1, b_2)}^{\lad} (q) =0$ if $(b_1', b_2') \not \leq (b_1, b_2)$. By \eqref{eq:P} and Proposition~\ref{prop:approx}, we have
\begin{align}
  \label{eq:P2}
(b_1 \diamondsuit_\zeta b_2)  (\xi_{-\lad} \otimes \eta_{\lad +\zeta} )
 = \sum_{(b'_1, b'_2)} 
P_{(b'_1, b'_2),(b_1, b_2)}  (q) \, 
(b'_1 \xi_{-\lad} \otimes b'_2 \eta_{\lad +\zeta} ) +o(1).
 \end{align}
Comparing \eqref{eq:Ptilde} and \eqref{eq:P2}, we conclude that 
\begin{align}
  \label{eq:P3}
\lim_{\lad \mapsto \infty}\limits  \widetilde P_{(b_1', b_2'),(b_1, b_2)}^{\lad} (q)
=P_{(b'_1, b'_2),(b_1, b_2)}  (q). 
\end{align}
Hence, we must have $P_{(b'_1, b'_2),(b_1, b_2)}(q) \in \Z[[q^{-1}]]$; moreover, $P_{(b'_1, b'_2),(b_1, b_2)}(q)=0$ unless $(b'_1, b'_2) \le (b_1, b_2)$, and  $P_{(b_1, b_2),(b_1, b_2)} (q)=1$. 

Now assume $\U$ is of ADE type. By \cite[Corollary 8.9]{We15}, $\widetilde P_{(b_1', b_2'),(b_1, b_2)}^{\lad} (q)$ in \eqref{eq:Ptilde} lies in $\N[q^{-1}]$, and then by \eqref{eq:P3}, we have $P_{(b_1', b_2'),(b_1, b_2)}^{\lad} (q) \in \N[[q^{-1}]] \cap \Q(q)$. 
\end{proof}

%%%%%%%%
%%%%%%%%
\section{PBW basis for $\Udot$}

In this section, we restrict ourselves to $\U$ of finite type; see however Remark~\ref{rem:affine}. We construct and study the PBW basis for $\Udot$ and its relation to the canonical basis. We present explicit formulas in rank 1.  

\subsection{PBW basis in finite type}

Let $\Phi^+$ be the set of positive roots and $W$ be the associated Weyl group for $\U$. 
Let $\preceq$ be a convex order on $\Phi^+$ associated to an (arbitrarily) fixed reduced expression of the longest element $w_0$ in the Weyl group $W$, and we denote the (decreasingly) ordered roots as $\Phi^+ =\{\beta_1, \ldots, \beta_N\}$, where $N=|\Phi^+|$. For each $\beta \in \Phi^+$, Lusztig \cite{Lu90, Lu93} defined a root vector $E_\beta$ using braid group action on $\U$, and its corresponding divided power $E_{\beta}^{(m)} =E_\beta^m / [m]_{q_\beta}!$, where $q_\beta =q_i$ if $\beta$ lies in the $W$-orbit of a simple root $\alpha_i$. 
We shall identify the set $\KP$ of Kostant partitions with $\N^{\Phi^+}$, thanks to the bijection given by ${\bf m} =(m_\beta)_{\beta \in \Phi^+} \in \N^{\Phi^+} \mapsto (\beta^{m_\beta})_{\beta \in \Phi^+}\in \KP$. 
%Denote by $\KP$ the set of all Kostant partitions ${\bf m} =(m_\beta)_{\beta \in \Phi^+} \in \N^{\Phi^+}$. 
 %${\bf m}:=(\beta_a^{m_a})_{1\le a \le N}$, for $m_a\in \N$. 
 For each $\bf m \in \KP$, we define
\[
E_{{\bf m}} := \prod_{a=1}^N E_{\beta_a}^{(m_{\beta_a})}.
\]
Then $\bold{W}^+:=\{E_{\bf m} | {\bf m} \in \KP \}$ forms a PBW basis for $\U^+$, which is orthogonal with respect to the bilinear form $(\cdot, \cdot)$ on $\U^+$ (see \cite[38.2.3]{Lu93}). 

The canonical basis for $\bf f$ can be parametrized by $\KP$ as well \cite{Lu90, Lu93}, and we shall write $\{b_{\bf m} | {\bf m} \in \KP \}$; via the isomorphism ${\bf f} \cong \U^+$, we obtain the canonical basis $\{b_{\bf m}^+ | {\bf m} \in \KP \}$ for $\U^+$ and $\{b_{\bf m}^- | {\bf m} \in \KP \}$ for $\U^-$. 
Each canonical basis element in $\U^+$ can be characterized by the following 2 properties:
\begin{enumerate}
\item[(i)]
 $\overline{b_{\bf m}^+} =b_{\bf m}^+$;
\item[(ii)]
 $b_{\bf m}^+ \in E_{\bf m} + \sum_{\bf m' \in \KP} q^{-1} \Z[q^{-1}] E_{\bf m'}$. 
\end{enumerate}
There is a suitable partial order $\preceq$ on $\KP$ (compatible with the convex order on $\Phi^+$), so that $b_{\bf m}^+ \in E_{\bf m} + \sum_{\bf m' \prec \bf m} q^{-1} \Z[q^{-1}] E_{\bf m'}$ \cite{Lu93} (also cf. \cite{Ka14, BKM14}). Note $\bf m' \preceq \bf m$ implies that $\sum_{\beta \in \Phi^+} m'_\beta \beta =\sum_{\beta \in \Phi^+} m_\beta \beta$.
Denote
\begin{align}
  \label{eq:bE}
b_{\bf m}^+ =\sum_{\bf m' \preceq \bf m} R_{{\bf m'}, {\bf m}}(q) E_{\bf m'}
\end{align}
where $R_{{\bf m'}, {\bf m}}(q) \in q^{-1} \Z[q^{-1}]$ for $\bf m' \prec \bf m$, and $R_{{\bf m}, {\bf m}}(q)=1$. 

Similarly, by applying suitable braid group actions to construct root vectors and their divided powers for $\U^-$ (which lie in $\U^-_{\Z[q^{-1}]}$), we obtain a PBW basis $\bold{W}^- := \{F_{\bf m} | {\bf m} \in \KP \}$ for $\U^-$, which is orthogonal with respect to the bilinear form $(\cdot, \cdot)$ on $\U^-$. The canonical basis $\{b_{\bf n}^- | {\bf n} \in \KP \}$ for $\U^-$ satisfies $b_{\bf n}^- \in F_{\bf n} + \sum_{\bf n' \prec \bf n} q^{-1} \Z[q^{-1}] F_{\bf n'}$. It follows that 
\begin{align}
  \label{eq:bF}
b_{\bf n}^- =\sum_{\bf n' \preceq \bf n} R_{{\bf n'}, {\bf n}}(q) F_{\bf n'}.
\end{align}

We formulate the following distinguished case of Theorem~\ref{thm:obasis} separately and a little more precisely, which is the original goal of this paper.

\begin{thm} [PBW basis for $\Udot$] 
 \label{thm:PBW}
Let $\U$ be finite type. Then the set
\[
\bold{W}\fus\bold{W} =\{ E_{\bf m} \fus_\zeta F_{\bf n} \mid  {\bf m}, {\bf n} \in \KP, \zeta \in X \} 
\]
forms an orthogonal basis (called a {\em PBW basis}) for $\Udot$. The {\em dual PBW basis} for $\Udot$ is given by 
\[
\Big \{ \prod_{\beta \in \Phi^+} (q_\beta^{-2}; q_\beta^{-2})_{m_\beta}(q_\beta^{-2}; q_\beta^{-2})_{n_\beta}E_{\bf m} \fus_\zeta F_{\bf n} \mid  {\bf m}, {\bf n} \in \KP, \zeta \in X \Big \}. 
\]
\end{thm}

\begin{proof}
Clearly $\bold{W}\fus\bold{W}$ is a basis for $\Udot$, by Theorem~\ref{thm:obasis}.
Its orthogonality follows by \eqref{eq:BF2} and the orthogonality of the PBW bases for $\U^{\pm}$ \cite[38.2.3]{Lu93}. Note that $(E_{\bf m}, E_{\bf m}) =\prod_{\beta \in \Phi^+} (q_\beta^{-2}; q_\beta^{-2})_{m_\beta}^{-1}$, cf. \cite{Lu93}. The dual PBW basis follows from this and \eqref{eq:BF2}. 
\end{proof}

Thanks to $E_{\bf m}  \fus_\zeta 1  = E_{\bf m}  \one_{\zeta}$ and $1 \fus_\zeta F_{\bf n} = F_{\bf n} \one_{\zeta}$, the PBW bases for $\U^\pm$ are parts of the PBW basis for $\Udot$. 

One can also formulate hybrid bases $\bold{W}\fus \B$ and $\B \fus \bold{W}$ for $\Udot$. 

\subsection{PBW basis vs (fused) canonical basis}

The fused canonical basis for $\Udot$ in finite type can now be written as $\B \fus \B =\{b_{\bf m} \fus_\zeta b_{\bf n} | {\bf m, \bf n} \in \KP, \zeta\in X \}$. We formulate its relation to the PBW basis $\bold{W}\fus\bold{W}$. 

\begin{prop}
  \label{prop:fCB-PBW}
Let $\bf m, \bf n \in \KP$ and $\zeta \in X$. We have 
\begin{align*}
b_{\bf m} \fus_\zeta b_{\bf n} &= E_{\bf m} \fus_\zeta F_{\bf n}  + 
\sum_{\stackrel{{\bf (m', n') \neq (m, n)}}{\bf m' \preceq \bf m, \bf n' \preceq \bf n}} R_{{\bf m'}, {\bf m}} (q) R_{{\bf n'}, {\bf n}} (q) \, 
E_{\bf m'} \fus_\zeta F_{\bf n'}.
 \end{align*}
Moreover, for the ADE type, we have $R_{{\bf m'}, {\bf m}} (q),  R_{{\bf n'}, {\bf n}} (q) \in \N[[q^{-1}]] \cap \Q(q)$. 
\end{prop}

\begin{proof}
By construction (see Proposition~\ref{prop:approx}), $x\fus_\zeta y$ is bilinear on $x$ and $y$. Thus the expansion formula for $b_{\bf m} \fus_\zeta b_{\bf n}$ follows by \eqref{eq:bE}--\eqref{eq:bF}. 

Assume $\U$ is of ADE type.  
Then we have $R_{{\bf m'}, {\bf m}} (q), R_{{\bf n'}, {\bf n}} (q) \in \N[q^{-1}]$. This was first proved in \cite[Corollary 10.7]{Lu90} when the reduced expressions of the longest element $w_0 \in W$ are ``adapted", and then proved by Syu Kato \cite{Ka14} in general (also see \cite{BKM14} and \cite{Oy18} for different approaches). 
%Thus, $P_{(b_{\bf m_1}, b_{\bf n_1}),(b_{\bf m}, b_{\bf n})}  (q) \, R_{{\bf m'}, {\bf m_1}} (q) R_{{\bf n'}, {\bf n_1}} (q)$ lies in $\N[[q^{-1}]] \cap \Q(q)$. 
\end{proof}

Now we formulate the relation between the canonical basis and PBW basis for $\Udot$.

\begin{cor}
  \label{cor:CB-PBW}
Let $\bf m, \bf n \in \KP$ and $\zeta \in X$. We have 
\begin{align*}
b_{\bf m} \diamondsuit_\zeta b_{\bf n} 
&= %E_{\bf m} \fus_\zeta F_{\bf n}  + 
\sum_{\stackrel{(b_{\bf m_1}, b_{\bf n_1}) \leq (b_{\bf m}, b_{\bf n})}{\bf m' \preceq \bf m_1, \bf n' \preceq \bf n_1}} 
P_{(b_{\bf m_1}, b_{\bf n_1}),(b_{\bf m}, b_{\bf n})}  (q) \,
R_{{\bf m'}, {\bf m_1}} (q) R_{{\bf n'}, {\bf n_1}} (q) \, 
E_{\bf m'} \fus_\zeta F_{\bf n'}.
 \end{align*}
Moreover, for the ADE type, $P_{(b_{\bf m_1}, b_{\bf n_1}),(b_{\bf m}, b_{\bf n})}  (q) \,
R_{{\bf m'}, {\bf m_1}} (q) R_{{\bf n'}, {\bf n_1}} (q) \in \N[[q^{-1}]] \cap \Q(q)$. 
\end{cor}

\begin{proof}
Follows by combining Theorem~\ref{thm:CB-fCB} and Proposition~\ref{prop:fCB-PBW}. 
\end{proof}

\begin{rem}  
% {Categorification}
Assume $\U$ is of ADE type. The (dual) PBW bases for $\U^+$ have been categorified in terms of (proper) standard modules over the KLR or quiver Hecke algebras \cite{Ka14, BKM14}. 
In light of the positivity properties of the transition matrices in Theorem~\ref{thm:CB-fCB}, Proposition~\ref{prop:fCB-PBW} and Corollary~\ref{cor:CB-PBW}, it is reasonable to ask for a categorification of the fused canonical basis and the PBW basis of $\Udot$. 
\end{rem}

\begin{rem}  
  \label{rem:affine}
Let $\U$ be of affine type. By a theorem of \cite{BCP99}, there exists a PBW basis $\bold{W}$ for $\bff$, which is actually a $\Z[q^{-1}]$-basis for the $\bff_{\Z[q^{-1}]}$; it gives rise to PBW bases for $\U^+$ and for $\U^-$. It still holds that the set $\bold{W} \fus \bold{W}$ forms a basis for $\Udot$ by Theorem~\ref{thm:obasis} and may be called its PBW basis. However, this basis is not orthogonal, as the PBW basis for $\U^+$ (or $\U^-$)  in \cite{BCP99} is not either. 
\end{rem}

%%%%%
%%%%%\section{PBW basis for $\Udot$ of rank one}
%  \label{sec:rank1}

%
%
\subsection{PBW basis in rank 1}
 \label{sec:rank1}

In the remainder of this section, we set $\U =\U_q(\mathfrak{sl}_2)$, the $\Qq$-algebra generated by $E$, $F$, $K, K^{-1}.$ We shall write $q_i=q$, identify the weight lattice $X =\Z$ and $X^+ =\N$. We work out explicit formulas of PBW basis via canonical basis for $\Udot$, and vice versa.

For $n\in \N$, let $L(n)$ be the highest weight $\U$-module with highest weight vector $\eta_n$, and let ${}^\omega L(n)$ be the lowest weight $\U$-module with lowest weight vector $\xi_{-n}$. For $m \in \Z$, we shall consider the tensor product $\U$-modules ${}^\omega L(p) \otimes L(m+p)$, for various $p\in \N$ such that $m+p\in \N$. 

Recall $\{\EE{a}|a\in \N\}$ is the PBW basis (and also the canonical basis) for $\U^+$, and $\{\FF{b}| b\in \N\}$ is the PBW basis (and the canonical basis) for $\U^-$. By Proposition~\ref{prop:limit} and Theorem~\ref{thm:PBW}, there exists a unique PBW basis element 
\[
\mf w_{m} (a,b):= \EE{a} \fus_{m} \FF{b} \in \Udot\one_{m}
\]
such that 
\begin{align}
  \label{tensor0}
\lim_{p\mapsto \infty}\limits  \Big( \mf w_m(a,b) (\xi_{-p} \otimes \eta_{p+m} ) - \EE{a} \xi_{-p} \otimes \FF{b} \eta_{p+m} \Big ) =0.
\end{align}
In particular, $\mf w_m(a,0) =\EE{a} \one_{m}$ and $\mf w_m(0,b) =\FF{b} \one_{m}$, for $a, b \in \N$. 

\begin{prop} 
 \label{prop:PBW1}
The set $\{\mf w_m(a,b) \mid a,b \in \N, m\in \Z\}$ forms an orthogonal (PBW) basis for $\Udot$ with respect to $(\cdot, \cdot)$. Moreover,  we have 
\begin{align}
 \Big(\mf w_m(a,b), \mf w_{m'} (a',b') \Big) &= \frac{\delta_{a,a'} \delta_{b,b'} \delta_{m,m'}}{(q^{-2}; q^{-2})_a  (q^{-2}; q^{-2})_b}.
  \label{norm}
\end{align}
\end{prop}
The dual PBW basis for $\Udot$ is  $\{(q^{-2}; q^{-2})_a  (q^{-2}; q^{-2})_b \mf w_m(a,b) \mid a,b \in \N, m\in \Z\}$.

\begin{proof}
The first statement is a rephrasing of Theorem~\ref{thm:PBW}. The explicit bilinear pairing formula \eqref{norm} follows from \eqref{eq:BF2} and the formula \eqref{BFhalf}.
\end{proof}

\subsection{PBW-expansion of  canonical basis in rank 1}

Recall the canonical basis for $\Udot$ consists of the following elements (cf. \cite{K93}, \cite[25.3.2]{Lu93}):
\[
 \EE{a} \FF{b} \one_{m}  \;\; (m \le b-a),
\qquad
\FF{b} \EE{a} \one_{m}  \;\; (m \ge b-a),
\qquad \text{ for $a, b \in \N$},
\]
with the (only) identification 
\begin{align}
  \label{EF=FE}
\EE{a} \FF{b} \one_{m} =\FF{b} \EE{a} \one_{m}, \quad \text{ for } m =b-a.
\end{align}

\begin{thm}
 \label{thm:PBW-CB}
Let $a, b \in \N$ and $m\in \Z$. Then we have 
\begin{align}
 \EE{a} \FF{b} \one_{m} 
 &= \sum_{s=0}^{\min(a,b)} \frac{q^{-s^2+s(a-b +m)}}{(q^{-2}; q^{-2})_s} \mf w_m(a-s,b-s),
 \quad \text{for } m \le b-a,
   \label{CB-W1} \\
  \FF{b} \EE{a} \one_{m} 
 &= \sum_{s=0}^{\min(a,b)}  \frac{q^{-s^2 \red{-} s(a-b +m)}}{(q^{-2}; q^{-2})_s} \mf w_m(a-s,b-s),
 \quad \text{for } m \ge b-a.
   \label{CB-W2}
\end{align}
\end{thm}
Note that the leading coefficients are 1 and all the (non-leading) coefficients in \eqref{CB-W1}--\eqref{CB-W2} lies in $q^{-1} \N[[q^{-1}]] \cap \Q(q)$. 

\begin{proof}
Assume $m\le b-a$. We consider the action of $\U$ on the module ${}^\omega L(p) \otimes L(m+p)$, for various $p\ge |m|$. Recall the following formula from \cite[\S25.3]{Lu93}, where we have replaced Lusztig's $\eta_{q}$ by $\eta_{p+m}$ here and his $(-n+2b)$ by $m$ here:
\begin{align}  \label{L25.3a}
&\EE{a} \FF{b} (\xi_{-p} \otimes \eta_{p+m}) 
\\
& =\sum_{s\ge 0; s \le a, s\le b} q^{s(a-s-p)} 
\qbinom{s-b+m+p}{s} \EE{a-s} \xi_{-p} \otimes \FF{b-s} \eta_{p+m}.
\notag
\end{align}
To make even more clear the dependence on $p$ of these coefficients, we rewrite the formula \eqref{L25.3a} as
\begin{align}
&\EE{a} \FF{b} (\xi_{-p} \otimes \eta_{p+m}) 
 \label{eq:EFpq} \\
& =\sum_{s\ge 0; s \le a, s\le b} q^{-s^2+s(a-b+m)} 
\prod_{d=1}^s \frac{1 -q^{2b-2m-2d-2p}}{1 -q^{-2d}}  \EE{a-s} \xi_{-p} \otimes \FF{b-s} \eta_{p+m}
\notag \\
& =\sum_{s\ge 0; s \le a, s\le b} \frac{q^{-s^2+s(a-b+m)} }{
(q^{-2}; q^{-2})_s}    \EE{a-s} \xi_{-p} \otimes \FF{b-s} \eta_{p+m}
 \notag \\
&\;  -\sum_{s\ge 0; s \le a, s\le b}
 \frac{
q^{-s^2+s(a-b+m)}}{(q^{-2}; q^{-2})_s} 
\big(1- (q^{2b-2m-2-2p};q^{-2})_s \big) \EE{a-s} \xi_{-p} \otimes \FF{b-s} \eta_{p+m}.
\notag
\end{align}
%Note that $s(a-b+m-s)<0$ for $s>0$, thanks to $m\le b-a$. 
Note that the first summand on the RHS\eqref{eq:EFpq} has coefficients independent of $p$; the second summand is an asymptotically-zero standard sequence as $p$ varies, and so it has limit $0$ as $p$ tends to $\infty$ by Lemma~\ref{lem:lim0}. Therefore, \eqref{CB-W1} follows from the identities \eqref{tensor0} and \eqref{eq:EFpq}.

Now assume $m\ge b-a$. Recall the following formula from \cite[\S25.3]{Lu93}:
\begin{align}
  \label{L25.3b}
&\FF{b} \EE{a} (\xi_{-p} \otimes \eta_{p+m}) 
\\
& =\sum_{s\ge 0; s \le a, s\le b} q^{s(b-s-m-p)} 
\qbinom{s-a+p}{s} \EE{a-s} \xi_{-p} \otimes \FF{b-s} \eta_{p+m}.
\notag
\end{align}
This can be rewritten as
\begin{align}
&\FF{b} \EE{a} (\xi_{-p} \otimes \eta_{p+m}) 
\label{eq:EFpq2} \\
& =\sum_{s\ge 0; s \le a, s\le b} q^{-s^2 \red{-}s(a-b+m)} 
\prod_{d=1}^s \frac{1 -q^{2a-2d-2p}}{1 -q^{-2d}}  \EE{a-s} \xi_{-p} \otimes \FF{b-s} \eta_{p+m}
\notag \\
& =\sum_{s\ge 0; s \le a, s\le b} \frac{q^{-s^2\red{-}s(a-b+m)} }{
(q^{-2}; q^{-2})_s}    \EE{a-s} \xi_{-p} \otimes \FF{b-s} \eta_{p+m}
 \notag \\
&\quad - \sum_{s\ge 0; s \le a, s\le b} 
\frac{ q^{-s^2\red{-}s(a-b+m)} }{(q^{-2}; q^{-2})_s} 
\big(1- (q^{2a-2-2p};q^{-2})_s \big)  \EE{a-s} \xi_{-p} \otimes \FF{b-s} \eta_{p+m}.
\notag
\end{align}
The second summand above has limit $0$ as $p$ tends to $\infty$, and \eqref{CB-W2} follows.
\end{proof}

\begin{cor}
  \label{cor:BF1}
The bilinear pairings between canonical basis elements are given by:
\begin{align*}
&\big(\EE{a} \FF{b} \one_{m}, \EE{a'} \FF{b'} \one_{m} \big)
\\
&=\sum_{s=\min(0,a-a')}^{\min(a,b)} \frac{%q^{-s^2+s(a-b +m) -(a'-a+s)^2 +(a'-a+s)(a-b+m)}
q^{-s^2 -(a'-a+s)^2 +(a'-a+2s)(a-b+m)} }{ 
(q^{-2}; q^{-2})_{a - s} (q^{-2}; q^{-2})_{b - s} (q^{-2}; q^{-2})_{s} (q^{-2}; q^{-2})_{a' - a + s}}
\notag
\end{align*}
if $m\le b-a =b'-a'$;
\begin{align*}
& \big(\FF{b} \EE{a} \one_{m}, \FF{b'} \EE{a'} \one_{m} \big)
\\
&=\sum_{s=\min(0,a-a')}^{\min(a,b)} \frac{%q^{-s^2-s(a-b +m) -(a'-a+s)^2 -(a'-a+s)(a-b+m)}
q^{-s^2 -(a'-a+s)^2 \red{-} (a'-a+2s)(a-b+m)} }{ 
(q^{-2}; q^{-2})_{a - s} (q^{-2}; q^{-2})_{b - s} (q^{-2}; q^{-2})_{s} (q^{-2}; q^{-2})_{a' - a + s}}
\notag
\end{align*}
if $m\ge b-a =b'-a'$;
The bilinear pairings are $0$ whenever $b-a \neq b'-a'$.
\end{cor}

\begin{proof}
This follows immediately from Theorem~\ref{thm:PBW-CB} and Proposition~\ref{prop:PBW1}. 
\end{proof}

\begin{rem}
The formulas in Corollary~\ref{cor:BF1} appeared first in \cite[Proposition 2.8]{La10} (whose convention differs from us by $q\leftrightarrow q^{-1}$), and their direct combinatorial proof is long and occupied \cite[\S10]{La10}. There is a very different formula for the bilinear pairing obtained by adjunctions, cf. \cite{La10}. 
\end{rem}

\begin{rem}
\label{rem:BFrank1}
An alternative short proof of Corollary~\ref{cor:BF1} is as follows (we focus on the case when $m \le b-a$ below). Use the formula \eqref{L25.3a} to compute the bilinear pairing 
\[
\Big(\EE{a} \FF{b} (\xi_{-p} \otimes \eta_{p+m}), \EE{a'} \FF{b'} (\xi_{-p} \otimes \eta_{p+m}) \Big)
\]
on the module ${}^\omega L(p) \otimes L(m+p)$ (which is easy as the summands in \eqref{L25.3a} are orthogonal). Then its limit in $\Q((q^{-1}))$ as $p$ tends to $\infty$ can be directly read off from the reformulation \eqref{eq:EFpq} of \eqref{L25.3a}.
This limit gives us $\big(\EE{a} \FF{b} \one_{m}, \EE{a'} \FF{b'} \one_{m} \big)$ according to  \cite[26.2.3]{Lu93}. 
%
%If one reflects on this argument, one can essentially see (and discover) the PBW basis for $\Udot$.
\end{rem}

\begin{rem}
The bilinear pairing on the modified $\imath$quantum group of rank one was computed \cite{Wa21} in the same approach as in Remark~\ref{rem:BFrank1}. For $\imath$quantum groups, this is the only way of computing for now as there is no characterization of the bilinear form via adjunction like \eqref{eq:ad1}--\eqref{eq:ad3}.
\end{rem}

\subsection{PBW basis via canonical basis in rank 1}
 \label{sec:rank1c}

Now we give a formula for the PBW basis in terms of canonical basis. 
\begin{thm}
For $a, b\in \N$ and $m\in \Z$,  we have 
\begin{align}
  \label{wEF}
\mf w_m(a,b) =  
 \begin{cases}
 \sum\limits_{s=0}^{\min(a,b)} (-1)^s \frac{q^{-s +s(a-b +m)}}{(q^{-2}; q^{-2})_s} \EE{a-s} \FF{b-s} \one_{m},
 & \quad \text{if } m \le b-a,
 \\ \\
 \sum\limits_{s=0}^{\min(a,b)} (-1)^s \frac{q^{-s \red{-}s(a-b +m)}}{(q^{-2}; q^{-2})_s} \FF{b-s} \EE{a-s} \one_{m},
 & \quad \text{if } m \ge b-a.
   \end{cases}
\end{align}
\end{thm}

\begin{proof}
  We shall prove the first formula for $\mf w_m(a,b)$ with $m \le b-a$; the proof of the other formula is entirely similar and will be skipped.
  
Let us fix $a,b \in \N$ and $m\in \Z$ such that $m \le b-a$.  One could rephrase the formula \eqref{CB-W1} as giving the unital triangular transition matrix from a basis $C :=\{\EE{a-s} \FF{b-s} \one_{m} \mid 0\le s\le \min(a,b) \}$ to another basis $P :=\{\mf w_m(a-s,b-s) \mid 0\le s\le \min(a,b)\}$ (these 2 sets have the same span). Accordingly,  the first formula in \eqref{wEF} gives the unital triangular transition matrix from $P$ to $C$, and we shall prove these two  transition matrices are inverses to each other. This boils down to verifying the following identity, for $k \ge 1$:
\[
\sum_{s=0}^k (-1)^{k-s} \frac{q^{-s^2+s(a-b+k)} q^{(k-s)(a-b+k-1)} }{(q^{-2}; q^{-2})_s (q^{-2}; q^{-2})_{k-s}} =0. 
\]
Upon multiplying with $(-1)^k q^{k(b-a-k+1)} (q^{-2}; q^{-2})_{k}$, we reduce the above identity to the following standard $q$-binomial identity, for $k\ge 1$  (cf. \cite[1.3.4]{Lu93}):
\[
\sum_{s=0}^k (-1)^{s} q^{s(1-m)} \qbinom{m}{s} =0.
\]
This proves the first formula in \eqref{wEF} (actually, we have shown that \eqref{CB-W1} and the first formula in \eqref{wEF} are equivalent). 
\end{proof}

The two different formulas of $\mf w_m(a,b)$ for $m=b-a$ in \eqref{wEF} coincide thanks to \eqref{EF=FE}. The simplest new PBW basis elements are given by 
\begin{align}
\label{eq:w11}
\mf w_m(1,1) =  
 \begin{cases}
EF \one_{m} - \frac{q^{m-1}}{1-q^{-2}} \one_{m} & \quad \text{if } m \le 0,
 \\ \\
FE \one_{m} - \frac{q^{-m-1} }{1-q^{-2}} \one_{m},
 & \quad \text{if } m \ge 0.
   \end{cases}
\end{align}
For $m\le 0$, we have
$\mf w_m(1,1) (\xi_{-p} \otimes \eta_{p+m})
= E\xi_{-p} \otimes F \eta_{p+m} -\frac{q^{-1-2p}}{1-q^{-2}} (\xi_{-p} \otimes \eta_{p+m}).$

%%%%%%%%%%%
%%%%%%%%%%%

\end{document}